\documentclass[reqno,a4paper]{amsart}
\usepackage{amsmath,amssymb}
\usepackage{graphicx}
\usepackage{ifthen}
\usepackage{mathtools}
\mathtoolsset{showonlyrefs,showmanualtags}
\usepackage{latexsym}
\usepackage{amsfonts}
\usepackage{mathrsfs}
\usepackage{amsthm}
\usepackage{nccmath}
\usepackage{tikz,xcolor}
\usepackage{extarrows}
\usepackage{shadethm}
\usepackage{ascmac}
\newdimen\tbaselineshift
\usetikzlibrary{intersections, calc, arrows.meta,positioning}
\theoremstyle{plain}
\newtheorem{theorem}{Theorem}[section]

\newtheorem{lemma}[theorem]{Lemma}

\newtheorem{remark}[theorem]{Remark}

\makeatletter
\renewenvironment{proof}[1][\proofname]{\par
  \pushQED{\qed}%
  \normalfont \topsep6\p@\@plus6\p@\relax
  \trivlist
  \item\relax
  {\bfseries
  #1\@addpunct{.}}\hspace\labelsep\ignorespaces
}{%
  \popQED\endtrivlist\@endpefalse
}
\makeatother

\newcommand{\mathsetintension}[2]{%
    \left\{#1\colon #2\right\}%
    }
\newcommand{\dlimit}{%
    \overset{d}{\to}%
    }

\title{Elephant random walk with polynomially decaying steps}
\author{Yuzaburo Nakano}
\address{Graduate~School~of~Engineering~Science, Yokohama~National~University, Yokohama, Japan}
\email{nakano-yuzaburo-zg@ynu.jp}
\date{}
\begin{document}
\begin{abstract}
    In this paper, we introduce a variation of {\color{black} the elephant random walk} whose steps are polynomially decaying. At each time {\mathversion{normal}$k$}, the walker's step size is $k^{-\gamma}$ with $\gamma>0$. We investigate effects of the step size exponent {\mathversion{normal}$\gamma$} and the memory parameter $\alpha\in [-1,1]$ on the long-time behavior of the walker. For fixed $\alpha$, {\color{black} it admits} phase transition from divergence to convergence (localization) at $\gamma_{c}(\alpha)=\max \{\alpha,1/2\}$. This means that large enough memory effect can shift the critical point for localization. Moreover, we obtain quantitative limit theorems which provide a detailed picture of the long-time behavior of the walker.
\end{abstract}
    \maketitle
\renewcommand{\theenumi}{(\roman{enumi})}%
\renewcommand{\labelenumi}{\theenumi}%
\renewcommand{\theenumii}{\alph{enumii})}%
\renewcommand{\labelenumii}{\theenumii}%
\section{Introduction}\label{sec:intro}
The most fundamental problem about random walks is to classify the long-time behavior of the walker. As one of the simplest random processes, we recall the recurrence behavior of the simple random walk (SRW) on the integers. A walker starts at zero, and at each step {\color{black}they flip} a coin and {\color{black}move} to the right if it comes up heads, otherwise {\color{black}move} to the left, where the coin comes up heads with probability $p$.
Let $S_{n}$ denote the position of the walker at time $n$. It is well known that if $p>1/2$ [resp. $p<1/2$], then $S_{n}$ diverges to $+\infty$ [resp. $-\infty$] almost surely (a.s.), while if $p=1/2$, then $S_{n}$ oscillates a.s., and actually {\color{black} they visit} every integer infinitely often a.s.
The former behavior is called transient, and the latter behavior recurrent. From a different perspective, we may regard $S_{n}$ as a sum {\color{black}$\sum_{k=1}^{n} X_{k}$} of independent identically distributed random variables $\{X_{k}\}$ with $P(X_{k}=1)=1-P(X_{k}=-1)=p$. By the above result, the random series {\color{black}$\sum_{k=1}^{n} X_{k}$} diverges with probability one for all $p\in [0,1]$ and its mode of divergence is classified. See Chapter 6 of Stout \cite{StoutAlmostSure} for the classical theory.

A natural generalization of the above question is the behavior of the random series {\color{black} $\Sigma_n := \sum_{k=1}^{n} c_kX_k$} for a fixed real valued sequence $\{c_k\}$. It is well known as the random signs problem (see Section 3.4 of Breiman \cite{BreimanProbability}). As  $\{\Sigma_n\}$ is divergent a.s. if $p \neq 1/2$, hereafter we assume that $p=1/2$. Rademacher \cite{Rademacher1922}, Khintchine and Kolmogorov \cite{KhintchineKolmogoroff1925} showed that $\{\Sigma_n\}$ converges a.s. if and only if {\color{black}$ \sum_{k=1}^{\infty} (c_k)^{2}<+\infty$}.  In the context of random walks, $c_k$ is the step size at time $k$, and thus the random walk $\{ \Sigma_n \}$ with a decreasing positive sequence $\{ c_k \}$ is called {\it a tired drunkard}, which can exhibit localization.  For $c_k =r^{k}$ with $r \in (0,1)$, the tired drunkard sleeps at some point a.s. If $r=1/2$, then the final resting place is uniformly distributed over the interval $[-1,1]$, and it is a long-standing problem to investigate the property of the distribution of the final resting place for other $r$ (see Section 24 of Padmanabhan \cite{PadmanabhanSleepingBeauties}). Another interesting choice is $c_k=k^{-\gamma}$ for $\gamma>0$, as the tired drunkard admits a phase transition: If $\gamma> 1/2$, then the walker eventually rests at some point a.s., while if $\gamma\le 1/2$, then the walker oscillates forever a.s.

Recently, random walks with long-memory have also attracted interests of many researchers. One of them is the elephant random walk (ERW), which is introduced by Sch{\"u}tz and Trimper \cite{SchutzTrimper2004Elephant} in 2004. It is a discrete-time nearest neighbour random walk on the integers with a complete memory of its whole history. We give a formal definition of ERW. The first step $X_1$ of the walker is $+1$ with probability $q\in [0,1]$, and $-1$ with probability $1-q$. For each $n=1,2,\dots$, let $U_n$ be uniformly distributed on $\{1,2,\dots,n\}$, and
\begin{equation}\label{eq:ERWstepdef}
    X_{n+1}=
    \begin{cases}
        X_{U_{n}} &  \text{with probability $p\in [0,1]$,}\\
        -X_{U_{n}} & \text{with probability $1-p$},
    \end{cases}
\end{equation}
where $\mathsetintension{U_n}{n=1,2,\dots}$ is an independent family of random variables. The sequence $\{X_k\}$ generates a one-dimensional random walk $\{T_n\}$ by
\begin{equation}
    T_0\equiv 0,\quad T_n\coloneq \sum_{k=1}^{n}X_k \quad \text{for $n=1,2,\dots$}.\nonumber
\end{equation}
Here $p$ is called the memory parameter. Note that if $p=q=1/2$, then $\{T_{n}\}$ is nothing but the symmetric SRW.

Let $\alpha\coloneq 2p-1$ and $\beta\coloneq 2q-1$. Sch{\"u}tz and Trimper \cite{SchutzTrimper2004Elephant} showed that
\begin{equation}\label{eq:expectation_of_T_n}
    E[T_{n}]=\beta a_n,
\end{equation}
where
\begin{equation}
    a_{0}\coloneq 1, \quad \text{and}\quad a_n\coloneq \prod_{k=1}^{n-1}\!\left(1+\frac{\alpha}{k}\right)\!=\frac{\Gamma(n+\alpha)}{\Gamma(n)\Gamma(\alpha+1)}\quad \text{for $n=1,2,\dots$}.\label{eq:a_n}
\end{equation}
By the Stirling formula for the Gamma functions,
\begin{equation}\label{eq:a_n_asymptotic}
    a_n \sim \frac{n^{\alpha}}{\Gamma(\alpha+1)}\quad \text{as $n\to \infty$}.\nonumber
\end{equation}
Here $x_n\sim y_n$ means that $x_n/y_n$ converges to $1$ as $n\to \infty$. In addition, Sch{\"u}tz and Trimper \cite{SchutzTrimper2004Elephant} showed that there are two distinct regimes about the mean square displacement according to the memory parameter $\alpha$:
\begin{equation}\label{eq:T_n_second_moment}
    E[(T_{n})^{2}]\sim 
    \begin{dcases}
        \frac{1}{1-2\alpha}n & \text{if $\alpha <1/2$},\\
        n\log n & \text{if $\alpha =1/2$},\\
        \frac{1}{(2\alpha -1)\Gamma(2\alpha)}n^{2\alpha} &\text{if $\alpha >1/2$}.
    \end{dcases}
\end{equation}
Based on this, the ERW is called diffusive if $\alpha<1/2$, and superdiffusive if $\alpha>1/2$.

Some years after their work, many authors \cite{BaurBertoin2016ERWPolyaurns,Bercu2018MrtingaleERW,ColettiGavaSchutz2017CLTforERW,ColettiGavaSchutz2017SIPforERW,KubotaTakei2019Gaussian,guerin2024fixedpointequationapproachsuperdiffusive} started to study limit theorems describing the influence of the memory parameter $\alpha$. 
We summarize principal results.
\begin{enumerate}
    \item\label{item:ERW_diffusive} If $\alpha<1/2$, then
    \begin{gather}
        \frac{T_{n}}{\sqrt{n}} \dlimit N\!\left(0,\frac{1}{1-2\alpha}\right)\!,\ \text{and}\ \limsup_{n\to \infty}\pm\frac{T_{n}}{\sqrt{2n\log\log n}}=\frac{1}{\sqrt{1-2\alpha}} \quad \text{a.s.}\label{eq:deffisive_LIL}
    \end{gather}
    Here $\dlimit$ denotes the convergence in distribution and $N(m,{\sigma}^{2})$ is a random variable having the normal distribution with mean $m$ and variance ${\sigma}^{2}$.
    \item\label{item:ERW_critical} If $\alpha =1/2$, then
    \begin{gather}
        \frac{T_{n}}{\sqrt{n\log n}}\dlimit N(0,1),\ \text{and}\ \limsup_{n\to\infty} \pm\frac{T_{n}}{\sqrt{2n\log n\log\log\log n}} =1 \quad \text{a.s.}\label{eq:critical_LIL}
    \end{gather}
    \item\label{item:ERW_superdiffusive} If $\alpha >1/2$, then
    \begin{align}\label{eq:superdiffusive_almost_sure_convergence}
        \lim_{n\to\infty} \frac{T_{n}}{n^{\alpha}}=L \quad \text{a.s. and in $L^{2}$}
    \end{align}
    with $P(L\neq 0)=1$. In addition, if $1/2<\alpha<1$, then
    \begin{equation}
        \frac{T_{n}-Ln^{\alpha}}{\sqrt{n}} \dlimit N\!\left(0,\frac{1}{2\alpha -1}\right)\!,\ \text{and}\ \limsup_{n\to\infty} \pm\frac{T_{n}-Ln^{\alpha}}{\sqrt{2n\log\log n}}=\frac{1}{\sqrt{2\alpha-1}} ~ \text{a.s.}\label{eq:superdiffusive_LIL}
    \end{equation}
\end{enumerate}

From \ref{item:ERW_diffusive}---\ref{item:ERW_superdiffusive} above,
\begin{equation}
    \mbox{for $\alpha <1$,} \quad \lim_{n \to \infty} \dfrac{T_n}{n} = 0 \ \mbox{a.s.,} \label{eq:ERW_LLN}
\end{equation}
which means that the asymptotic speed of $T_{n}$ is $0$ for any $\alpha<1$. Still the ERW admits a phase transition from recurrence to transience at the critical value $\alpha=1/2$ {\color{black}(see \cite{Qin2023recurrence} for the recurrence result in $d$-dimensional lattices)}. From \ref{item:ERW_diffusive}, the behavior of $\{T_{n}\}$ for $\alpha<1/2$ is quite similar to that of the symmetric SRW ($\alpha=\beta=0$). In the superdiffusive case $\alpha\in (1/2,1)$, although $L$ is in \ref{item:ERW_superdiffusive} is non-Gaussian, $Ln^{\alpha}$ should be regarded as ``random drift'' produced by the influence of long-memory, and the fluctuation from it is still Gaussian. The intermediate behavior is observed in the critical case \ref{item:ERW_critical}.

In this paper, we consider a generalization of ERW whose step sizes are polynomially decaying. Our model is defined as follows. Let $\{X_{k}\}_{k\ge 1}$ be the steps of ERW defined by \eqref{eq:ERWstepdef}. The elephant random walk with polynomially decaying steps $\{S_{n}\}$ is
\begin{equation}\label{eq:defineERWdecay}
    S_{0}\coloneq 0, \quad S_{n}\coloneq \sum_{k=1}^{n}\frac{X_{k}}{k^{\gamma}} \quad \text{for $n=1,2,\dots$},
\end{equation}
with $\gamma >0$. Note that if $\gamma=0$, then $\{S_{n}\}$ is the original ERW.

Our paper deals with the almost sure long-time behavior of the walker. For fixed $\alpha\in [-1,1]$, as $\gamma$ increases, $\{S_{n}\}$ admits a phase transition from divergence to convergence as the critical value $\gamma_{c}=\gamma_{c}(\alpha)\coloneq \max \{\alpha, 1/2\}$. Moreover, we give the classification of the modes of divergence of $\{S_{n}\}$. If $\alpha\le 1/2$ and $\gamma<\gamma_{c}(\alpha)=1/2$, then $\{S_{n}\}$ oscillates a.s. like the symmetric SRW with polynomially decaying steps. On the other hand, if $\alpha>1/2$ and $\gamma \le \gamma_{c}(\alpha)=\alpha$, then $\{S_{n}\}$ diverges to $+\infty$ or $-\infty$ a.s.


Recently there have been many studies on variations of the ERW. Somewhat similar settings
to ours are the ERW with random step sizes (see \cite{DedeckerFanHuMerlevede2023CLTrandomstepsizes,FanShao2022Cramer,Zhang2024does}), and step-reinforced/counterbalanced random walks (see e.g. \cite{Busingershark,Bertoinstepreinforced,Bertoinnoisereinforced,Bertoincounterbalancing,HuZangSLTstepreinforced,hua2024strongapproximationssurecentral}). Unlike those models, the ERW with polynomially decaying steps can localize, which is the principal novelty of our model.

{\color{black} In the higher dimensional case, the walker is expected to exhibit more complicated behaviour depending not only on $\alpha$ and $\gamma$ but also on the spatial dimension. This is one of very important future problems. In this paper, we would like to focus on one-dimensional case and give rather complete picture of phase transition, with several limit theorems.}
\section{Main results}\label{sec:results}
Out first theorem describes the quantitative behavior of the ERW with polynomially decaying steps $\{S_{n}\}$, defined by \eqref{eq:defineERWdecay}.
\begin{theorem}\label{thm:main}
    \begin{enumerate}
        \item If $\alpha\in [-1,1/2]$, then
        \begin{equation}
            P\!\left(-\infty =\liminf_{n\to\infty}S_{n}<\limsup_{n\to\infty}S_{n}=+\infty\right)=1 \label{eq:subdiverge}
        \end{equation}
        for any $\gamma\in (0,1/2]$ with $\displaystyle \gamma\neq \gamma_{0}(\alpha)\coloneq \max \{ \alpha, -\alpha/(1-2\alpha)\}$. On the other hand, $\{S_{n}\}$ converges with probability one for $\gamma >1/2$. 
        \item If $\alpha\in (1/2,1]$, then
        \begin{equation}
            P\!\left(\lim_{n\to\infty} S_{n}=-\infty\ \text{or}\ \lim_{n\to\infty}S_{n}=+\infty\right)=1 \label{eq:superdiver}
        \end{equation}
        for any $\gamma\in (0,\alpha]$, while $\{S_{n}\}$ converges with probability one for $\gamma >\alpha$.
    \end{enumerate}
\end{theorem}
For a summary of the above theorem, see Fig. \ref{fig:phase_diagram}.
\begin{remark}
    {\color{black} Information about the distribution of the limiting random variable $S_{\infty}$ for $\gamma>\gamma_{c}(\alpha)$ is scarce. This remains a long-standing problem even for the case $\alpha=0$, where $S_{n}$ is a sum of independent random variables. In Appendix \ref{sec:A1}, we show that $S_{n}\to S_{\infty}$ in $L^{2}$ if $\gamma>\gamma_{c}(\alpha)$. Thus, we can obtain semi-explicit formulae for the average and the second moment of $S_{\infty}$.}
\end{remark}
\begin{figure}[h]
    \centering
    \begin{tikzpicture}
        \node[below] (O) at (0,0) {O}; \node[below] (L) at (-4.8,0) {$-1$};
        \node[below] (R) at (4.8,0) {$1$}; \node[below] (M) at (2.4,0) {$1/2$};
        \node[above left] (N) at (0,2.4) {$1/2$};
        \draw (2.4,2.4) -- (-4.8,2.4);
        \draw (-4.8,0) -- (-4.8,7.2); \draw (4.8,0)--(4.8,7.2);
        \draw (2.376,0) -- (2.376,2.4); \draw[dotted,line width=1pt] (2.448,0) -- (2.448,2.4);
        \draw[->,>=stealth] (-6.72,0) -- (6.72,0)node[right] {$\alpha$};
        \draw[->,>=stealth] (0,0) -- (0,7.2)node[below right] {$\gamma$};
        \draw[dashed] plot[domain=0:2.4](\x,\x); \draw plot[domain=2.4:4.8](\x,\x) node[right] {$\gamma=\alpha$}; 
        \draw[dashed] plot[domain=0:-4.8,smooth](\x,{-\x/(1-(\x /2.4))}) node[left] {$\displaystyle \gamma=\frac{-\alpha}{1-2\alpha}$};
    
        \node (a) at (-2.4,4.8) {convergent};  \node (b) at (2.4,4.8) {convergent};
        \node[align = left] (d) at (-3.36,0.48) {oscillatory}; \node[align=left] (c) at (-1.2,1.68) {oscillatory};
        \node[align=center] (e) at (3.6,1.68) {$\displaystyle \lim_{n\to\infty} S_{n}=+\infty$\\or \\$\displaystyle \lim_{n\to\infty} S_{n}=-\infty$};
        \node[align = left] (f) at (0.84,1.776) {oscillatory}; \node[align=left] (g) at (1.44,0.48) {oscillatory};
    \end{tikzpicture}
    \caption{The classification of the long-time behavior of $\{S_{n}\}$.}
    \label{fig:phase_diagram}
\end{figure}
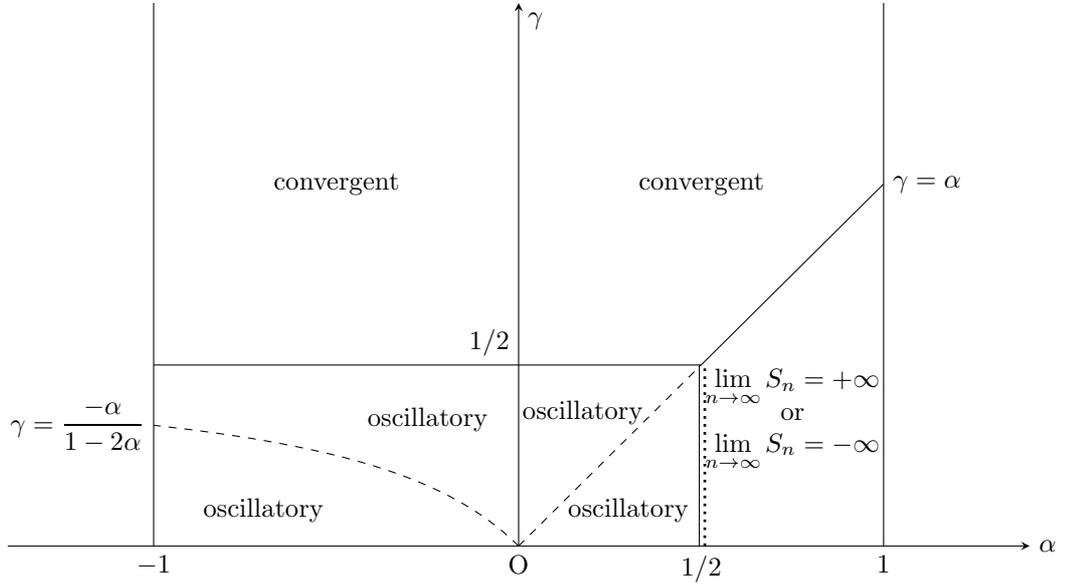

As a quantitative result, we have the central limit theorem (CLT) and the law of the iterated logarithm (LIL) for $\{S_{n}\}$.
\begin{theorem}\label{thm:Limittheorems}
    \begin{enumerate}
        \item \label{enum:Thm2_alphasub}Suppose that $\alpha\in [-1,1/2)$.
        \begin{enumerate}
            \item \label{enum:Thm2_alpahsub_gammasub}For any $\gamma \in (0,1/2)$ with $\displaystyle \gamma\neq \gamma_{0}(\alpha)$, there exists positive numbers $c_{1}(\alpha,\gamma)$ and $c_{2}(\alpha,\gamma)$ depending only on $\alpha$ and $\gamma$ such that
            \begin{equation}
                c_{1}(\alpha,\gamma)\le \limsup_{n\to\infty}\pm\frac{S_{n}}{\sqrt{2n^{1-2\gamma}\log\log n}}\le c_{2}(\alpha,\gamma) \quad \text{a.s.} \nonumber
            \end{equation}
            \item \label{enum:Thm2_alphasub_gammacritical}For $\gamma=1/2$, $\displaystyle \frac{S_{n}}{\sqrt{\log n}}\dlimit N\!\left(0,\frac{1}{(1-2\alpha)^{2}}\right)$ and
            \begin{equation}
                \limsup_{n\to\infty} \pm\frac{S_{n}}{\sqrt{2\log n\log\log\log n}}=\frac{1}{1-2\alpha}\quad \text{a.s.}\nonumber
            \end{equation}
        \end{enumerate}
        \item \label{enum:Thm2_alphacritical}Suppose that $\alpha=1/2$. For $\gamma\in (0,1/2)$, $\displaystyle \frac{S_{n}}{\sqrt{n^{1-2\gamma}\log n}}\dlimit N\!\left(0,\frac{1}{(1-2\gamma)^{2}}\right)$ and 
        \begin{equation}
           \limsup_{n\to\infty} \pm\frac{S_{n}}{\sqrt{2n^{1-2\gamma}\log n\log\log\log n}}=\frac{1}{1-2\gamma}\quad \text{a.s.}\nonumber
        \end{equation}
        \item \label{enum:Thm2_alphasuper}Suppose that $\alpha\in (1/2,1]$. Here $L$ is the random variable defined by \eqref{eq:superdiffusive_almost_sure_convergence}.
        \begin{enumerate}
            \item \label{enum:Thm2_alphasuper_gammasub}For \(\gamma\in (0,\alpha)\), $\displaystyle \lim_{n\to\infty} \frac{S_{n}}{n^{\alpha -\gamma}}=\frac{\alpha L}{\alpha-\gamma}$ with probability one. Moreover, for the fluctuation of $S_{n}$ from $\displaystyle \frac{\alpha L}{\alpha-\gamma}n^{\alpha-\gamma}$, the followings hold.
            \begin{itemize}
                \item If $\gamma\in (0,1/2)$ and $\displaystyle \gamma\neq \frac{-\alpha+\alpha\sqrt{\alpha^{2}+2\alpha-1}}{2\alpha-1}$, then there exists positive constants $c_{3}(\alpha,\gamma)$ and $c_{4}(\alpha,\gamma)$ such that
                \begin{equation}
                    c_{3}(\alpha,\gamma)\le \limsup_{n\to\infty}\frac{S_{n}-\frac{\alpha L}{\alpha-\gamma}n^{\alpha-\gamma}}{\sqrt{2n^{1-2\gamma}\log\log n}}\le c_{4}(\alpha,\gamma) \quad \text{a.s.}\nonumber
                \end{equation}
                \item If $\gamma=1/2$, then $\displaystyle \limsup_{n\to\infty}\pm\frac{S_{n}-\frac{\alpha L}{\alpha-\gamma}n^{\alpha-\gamma}}{\sqrt{2\log n\log\log\log n}}=1$ a.s.
                \item If $\gamma\in (1/2,\alpha)$, then the sequence $\left\{S_{n}-\frac{\alpha L}{\alpha-\gamma}n^{\alpha-\gamma}\right\}$ is bounded a.s.
            \end{itemize}
            \item \label{enum:Thm2_alphasuper_gammacritical}For \(\gamma =\alpha\), $\displaystyle \lim_{n\to\infty} \frac{S_{n}}{\log n}=\alpha L$ with probability one.
            \item \label{enum:Thm2_alphasuper_gammasuper}For $\gamma>\alpha$, $\{S_{n}\}$ converges almost surely. Letting $\displaystyle S_{\infty}\coloneq \lim_{n\to\infty}S_{n}$ a.s., we have
            \begin{equation}
                P\!\left(-\infty<\liminf_{n\to\infty}\frac{S_{\infty}-S_{n}}{n^{\alpha-\gamma}}\le \limsup_{n\to\infty}\frac{S_{\infty}-S_{n}}{n^{\alpha-\gamma}}<+\infty\right)=1.\label{eq:superdiffusive_gammasuper}
            \end{equation}
        \end{enumerate}
    \end{enumerate}
\end{theorem}
\begin{remark}
    {\color{black} Our proof of Theorem 2 is based on Equation \eqref{eq:Snineq} below. To obtain CLT and LIL for $S_{n}$, we have to treat a sum of two dependent random variables, and that is the reason why our LIL are weaker than usual, and our CLT is restricted to a specific case. A similar problem arises also for the ERW with random step sizes (see \cite{DedeckerFanHuMerlevede2023CLTrandomstepsizes}). We need to establish a new approach to deal with such problems. The restriction $\gamma\neq \gamma_{0}(\alpha)$ for $\alpha<0$ will be circumvented by this. On the other hand, another problem arises when $\alpha>0$ for $\gamma=\gamma_0(\alpha)$, since the crucial Equation \eqref{eq:Snineq} degenerates.
}
\end{remark}
\begin{remark}
    {\color{black} The behavior of $\sum_{k=1}^{n}c_{k}X_{k}$ for general coefficients $c_{k}$ is intended as a subject for future studies. Some of our proofs work for $c_{k}\sim k^{-\gamma}$ as well, but such generalization might affect the behavior below or near the critical line.}
\end{remark}
\section{Proofs}\label{sec:proofs}
Let $\mathcal{F}_{0}$ be the trivial $\sigma$-field, $\mathcal{F}_n$ be the $\sigma$-field generated by $X_1,\dots,X_n$, and $H_{n}\coloneq \# \mathsetintension{1\le j\le n}{X_j=+1}$. For $n=1,2,\dots$, the conditional distribution of $X_{n+1}$ given the history up to time $n$ is
\begin{align}\label{eq:Xnconditionaldistr}
    P(X_{n+1}=+ 1\mid \mathcal{F}_n)&=\frac{H_n}{n}\cdot p+\!\left(1-\frac{H_n}{n}\right)\!\cdot (1-p)\nonumber\\
    &=\alpha\cdot \frac{H_n}{n}+(1-\alpha)\cdot \frac{1}{2},\nonumber
\end{align}
and the conditional expectation of $X_{n+1}$ is
\begin{equation}
    E[X_{n+1}\mid \mathcal{F}_{n}]=P(X_{n+1}=+1\mid \mathcal{F}_n)-P(X_{n+1}=-1 \mid \mathcal{F}_n)=\alpha\cdot \frac{T_{n}}{n}.\label{eq:Xnconditionalexp}
\end{equation}
Thus, we have
\begin{align}
    E[T_{n+1}\mid \mathcal{F}_n]&=E[T_{n} + X_{n+1}\mid \mathcal{F}_n]=\!\left(1+\frac{\alpha}{n}\right)\!T_{n}.\nonumber
\end{align}
To analyze the long-time behavior of \(S_{n}\), we use the Doob decomposition:
\begin{align}\label{eq:hatSn_decomposition}
    S_{n}&=\sum_{k=1}^{n}\frac{X_{k}-E[X_{k}\mid \mathcal{F}_{k-1}]}{k^{\gamma}}+\sum_{k=1}^{n}\frac{E[X_{k}\mid \mathcal{F}_{k-1}]}{k^{\gamma}}\eqcolon M_{n}+A_{n}.
\end{align}
We give the proofs of the main results in separate subsections. In Section \ref{subsec:martingale} we prove limit theorems for $\{M_{n}\}$ using a standard martingale limit theory. A useful expression of $\{A_{n}\}$ in terms of $\{S_{n}\}$ and $\{T_{n}\}$ will be given in Section \ref{subsec:An}. Theorems \ref{thm:Limittheorems} and \ref{thm:main} will be proved in Sections \ref{subsec:ProofofThm2} and \ref{subsec:proof_main}.
\subsection{Limit theorems for the martingale part $\{M_{n}\}$}\label{subsec:martingale}
For the martingale part $\{M_{n}\}$, we have the following CLT and LIL.
\begin{theorem}\label{thm:Mn}
    Suppose that $\alpha\in [-1,1)$.
    \begin{enumerate}
        \item \label{thm:Mn_sub}If $\gamma<1/2$, then
        \begin{center}
            $\displaystyle \frac{M_{n}}{n^{1/2-\gamma}}\dlimit N\!\left(0,\frac{1}{1-2\gamma}\right)$, and $\displaystyle \limsup_{n\to\infty} \pm \frac{M_{n}}{\sqrt{2n^{1-2\gamma}\log\log n}}=\frac{1}{\sqrt{1-2\gamma}}$ a.s.
        \end{center}
        \item \label{thm:Mn_critical}If $\gamma=1/2$, then
        \begin{center}
            $\displaystyle \frac{M_{n}}{\sqrt{\log n}}\dlimit N\!\left(0,1\right)$, and $\displaystyle \limsup_{n\to\infty} \pm \frac{M_{n}}{\sqrt{2\log n\log\log\log n}}=1$ a.s.
        \end{center}
        \item \label{thm:Mn_super}If $\gamma>1/2$, then
        \begin{center}
            $\displaystyle \frac{M_{n}-M_{\infty}}{n^{1/2-\gamma}}\dlimit N\!\left(0,\frac{1}{2\gamma-1}\right)$, and $\displaystyle \limsup_{n\to\infty} \pm \frac{M_{n} - M_{\infty}}{\sqrt{2n^{1-2\gamma}\log\log n}}=\frac{1}{\sqrt{2\gamma-1}}$ a.s.,
        \end{center}
        where $\displaystyle M_{\infty}\coloneq \lim_{n\to\infty}M_{n}$ with probability one and in $L^{2}$. The random variable $M_{\infty}$ has a positive variance.
    \end{enumerate}
\end{theorem}
The rest of this subsection is devoted to the proof of Theorem \ref{thm:Mn}. Let
\begin{equation}
    d_{k}\coloneq M_{k}-M_{k-1}=\frac{X_{k}-E[X_{k}\mid \mathcal{F}_{k-1}]}{k^{\gamma}} \quad \text{for \(k=1,2,\dots\),}\label{eq:martingaledifferences}
\end{equation}
where \(M_{0}\coloneq 0\). Note that \(\displaystyle |d_{k}|\le 2k^{-\gamma}\) since \(|X_{k}|=1\).
\begin{lemma}\label{lemma:hatMnmartingale}
    The sequence \(\{M_n\}\) is a square-integrable martingale with mean \(0\).
\end{lemma}
\begin{proof}
    By the definition of \(d_k\) by \eqref{eq:martingaledifferences}, we have \(E[d_{k}\mid \mathcal{F}_{k-1}]=0\) for \(k=1,2,\dots\). Moreover,
    \begin{align}
        E[(d_{k})^{2}\mid \mathcal{F}_{k-1}]&=\frac{E[(X_{k}-E[X_{k}\mid \mathcal{F}_{k-1}])^{2} \mid \mathcal{F}_{k-1} ]}{k^{2\gamma}} \nonumber\\
        &= \frac{E[(X_{k})^{2}\mid \mathcal{F}_{k-1}] - (E[X_{k}\mid \mathcal{F}_{k-1}])^{2}}{k^{2\gamma}}= \frac{1 - (E[X_{k}\mid \mathcal{F}_{k-1}])^{2}}{k^{2\gamma}}.\nonumber
    \end{align}
    Since \(|X_{k}|=1\), we have \(\displaystyle E[M_{n}^{2}]=\sum_{k=1}^{n}E[(d_{k})^{2}]<+\infty\) for each \(n\).
\end{proof}
For \(n=1,2,\dots\), let
\begin{equation}\label{eq:snVnUn_definition}
    s_{n}^{2}\coloneq \sum_{k=1}^{n} E[(d_{k})^{2}],\quad V_{n}^{2}\coloneq \sum_{k=1}^{n} E[(d_{k})^{2}\mid \mathcal{F}_{k-1}],\quad U_{n}^{2}\coloneq \sum_{k=1}^{n} (d_{k})^{2},\nonumber
\end{equation}
and
\begin{center}
    $\displaystyle s_{\infty}^{2}\coloneq \lim_{n\to\infty}s_{n}^{2}$, $\displaystyle V_{\infty}^{2}\coloneq \lim_{n\to\infty}V_{n}^{2}$ a.s. and $\displaystyle U_{\infty}^{2}\coloneq \lim_{n\to\infty}U_{n}^{2}$ a.s.
\end{center}
whenever these limits exist.
\begin{lemma}\label{lemma:snVnUn}
    Suppose that $\alpha\in [-1,1)$.
    \begin{enumerate}
        \item If $\gamma\le 1/2$, then $ V_{n}^{2}\sim s_{n}^{2}\sim \sum_{k=1}^{n} k^{-2\gamma}$ and $U_{n}^{2}-V_{n}^{2}=o(s_{n}^{2})$ as $n\to \infty$ almost surely.\label{lemma:snVnUn_gamma_sub}
        \item If $\gamma>1/2$, then $\{U_{n}^{2}\}$, $\{V_{n}^{2}\}$ and $\{s_{n}^{2}\}$ converge almost surely. Moreover, we have $ \hat{V}_{n}^{2}\sim \hat{s}_{n}^{2}\sim \sum_{k=n}^{\infty} k^{-2\gamma}$ and $\hat{U}_{n}^{2}-\hat{V}_{n}^{2}=o(\hat{s}_{n}^{2})$ as $n\to \infty$ almost surely, where $\hat{s}_{n}^{2}\coloneq s_{\infty}^{2}-s_{n}^{2}$, $\hat{V}_{n}^{2}\coloneq V_{\infty}^{2}-V_{n}^{2}$ and $\hat{U}_{n}^{2}\coloneq U_{\infty}^{2}-U_{n}^{2}$.\label{lemma:snVnUn_gamma_super}
    \end{enumerate}
\end{lemma}
\begin{proof}
    \ref{lemma:snVnUn_gamma_sub} Suppose that \(\gamma \le 1/2\). By \eqref{eq:ERW_LLN} and \eqref{eq:Xnconditionalexp}, we have
    \begin{equation}\label{eq:conditionalexpectation_asymptotic}
        E[(d_{k})^{2}\mid \mathcal{F}_{k-1}]=\frac{1-(E[X_{k}\mid \mathcal{F}_{k-1}])^{2}}{k^{2\gamma}} \sim \frac{1}{k^{2\gamma}} \quad \text{as $k\to \infty$}\nonumber
    \end{equation}
    with probability one. Moreover, by \eqref{eq:T_n_second_moment} and \eqref{eq:Xnconditionalexp}, we obtain
    \begin{equation}
        E[(d_{k})^{2}]=\frac{1-E[(E[X_{k}\mid \mathcal{F}_{k-1}])^{2}]}{k^{2\gamma}}\sim \frac{1}{k^{2\gamma}}\quad \text{as $k\to \infty$}.\nonumber
    \end{equation}
    Thus, since $E[(d_{k})^{2}\mid \mathcal{F}_{k-1}]\sim E[(d_{k})^{2}]$ as $k\to\infty$, we have, with probability one,
    \begin{equation}\label{eq:s_n^2_asymptotic}
        V_{n}^{2}\sim s_{n}^{2}\sim \sum_{k=1}^{n}\frac{1}{k^{2\gamma}}\quad \text{as $n\to\infty$.}
    \end{equation}
    To prove $U_{n}^{2}-V_{n}^{2}=o(s_{n}^{2})$ a.s., by Kronecker's lemma, it suffices to show
    \begin{equation}\label{eq:sum_convergence}
        \sum_{k=1}^{\infty}\frac{1}{s_{k}^{2}}\{(d_{k})^{2}-E[(d_{k})^{2}\mid \mathcal{F}_{k-1}]\}\quad \text{converges a.s.}
    \end{equation}
    Letting
    \begin{equation}
        \widehat{d}_{n}\coloneq \frac{1}{s_{n}^{2}}\{(d_{n})^{2}-E[(d_{n})^{2}\mid \mathcal{F}_{n-1}]\},\quad  m_{n}\coloneq \sum_{k=1}^{n}\widehat{d}_{k}\quad \text{and}\quad m_{0}\coloneq 0,\nonumber
    \end{equation}
    $\{m_{n}\}$ is a martingale with mean zero. We now show that $\{m_{n}\}$ is $L^{2}$-bounded, i.e.
    \begin{equation}\label{eq:martingale2_sum}
        \sup_{n\ge 1}E[m_{n}^{2}]=\sum_{k=1}^{\infty}E[(\widehat{d}_{k})^{2}]<+\infty,
    \end{equation}
    which together with Doob's convergence theorem (Corollary 2.2 in \cite{HallHeyde11980Martingale}) yields \eqref{eq:sum_convergence}. Since
    \begin{align}
        E[(\widehat{d}_{n})^{2}\mid \mathcal{F}_{n-1}]&=E\left[\frac{1}{s_{n}^{4}}\{(d_{n})^{2}-E[(d_{n})^{2}\mid \mathcal{F}_{n-1}]\}^{2}\mid \mathcal{F}_{n-1}\right]\nonumber\\
        &=\frac{1}{s_{n}^{4}}\{E[(d_{n})^{4}\mid \mathcal{F}_{n-1}]-(E[(d_{n})^{2}\mid \mathcal{F}_{n-1}])^{2}\}\nonumber\\
        &\le \frac{1}{s_{n}^{4}}E[(d_{n})^{4}\mid \mathcal{F}_{n-1}]\quad \text{a.s.},\nonumber
    \end{align}
    we have $E[(\widehat{d}_{n})^{2}]\le s_{n}^{-4}E[(d_{n})^{4}]$. By \eqref{eq:s_n^2_asymptotic} and $\displaystyle |d_{k}|\le 2k^{-\gamma}$,
\begin{equation}\label{eq:sn^4_conditionalEXP}
    \frac{1}{s_{n}^{4}}E[(d_{n})^{4}]\le \frac{1}{s_{n}^{4}}\cdot\frac{16}{n^{4\gamma}}\sim
    \begin{dcases*}
        \frac{16(1-2\gamma)^{2}}{n^{2}} &  if $\gamma<1/2$,\\
        \frac{16}{(n\log n)^{2}} & if $\gamma=1/2$,
    \end{dcases*}
\end{equation}
as $n\to\infty$. Thus, we have
\begin{equation}
    \sum_{n=1}^{\infty}\frac{1}{s_{n}^{4}}E[(d_{n})^{4}]<+\infty,\label{eq:sn^4_expectation}
\end{equation}
which implies \eqref{eq:martingale2_sum}.

\ref{lemma:snVnUn_gamma_super} By considering $\{\hat{s}_{n}^{2}\}$, $\{\hat{V}_{n}^{2}\}$ and $\{\hat{U}_{n}^{2}\}$, instead of $\{{s}_{n}^{2}\}$, $\{{V}_{n}^{2}\}$ and $\{{U}_{n}^{2}\}$, respectively, we can give the proof of \ref{lemma:snVnUn_gamma_super} in the same way as \ref{lemma:snVnUn_gamma_sub}.
\end{proof}

\begin{proof}[Proof of Theorem \ref{thm:Mn}]
    We check the conditions of Theorem 1 in \cite{Heyde1977MartingaleLimitTheorem}. Suppose that $\gamma\le 1/2$. In that case, $s_{n}^{2}\to \infty$ as $n\to\infty$. By Lemma \ref{lemma:snVnUn} \ref{lemma:snVnUn_gamma_sub}, we have $\displaystyle s_{n}^{-2}U_{n}^{2}\to 1$ as $n\to\infty$ a.s. Since $\displaystyle (d_{k})^{2}\le 4k^{-2\gamma}$, 
    \begin{center}
        $\displaystyle s_{n}^{-2}E\!\left[\sup_{1\le k \le n} (d_{k})^{2}\right]\!\to 0$ as $n\to \infty$ with probability one.
    \end{center}
Thus, 
    \begin{center}
        $\displaystyle \frac{M_{n}}{n^{1/2-\gamma}}\dlimit N\!\left(0,\frac{1}{1-2\gamma}\right)$ if $\gamma<1/2$, and $\displaystyle \frac{M_{n}}{\sqrt{\log n}}\dlimit N\!\left(0,1\right)$ if $\gamma=1/2$.
    \end{center}
    For any $\varepsilon>0$, as
    \begin{equation}
        \frac{1}{s_{k}}E[|d_{k}|\colon |d_{k}|>\varepsilon s_{k}]\le \frac{1}{\varepsilon^{3}s_{n}^{4}}E[(d_{k})^{4}],\nonumber
    \end{equation}
    we obtain, by \eqref{eq:sn^4_expectation},
    \begin{center}
        $\displaystyle \sum_{k=1}^{\infty} \frac{1}{s_{k}}E[|d_{k}|\colon |d_{k}|>\varepsilon s_{n}]<\infty$.\nonumber
    \end{center}
    Thus, writing $\phi(t)\coloneq (2t\log\log (t\vee 3))^{1/2}$, we have
    \begin{equation}
        \limsup_{n\to\infty} \pm\frac{M_{n}}{\phi(U_{n})}=1 \quad \text{a.s.},\nonumber
    \end{equation}
    which implies the law of the iterated logarithm for $\{M_{n}\}$ in the case $\gamma\le 1/2$.

    Suppose that $\gamma>1/2$. By Lemma \ref{lemma:snVnUn} \ref{lemma:snVnUn_gamma_super} and Doob's convergence theorem,
    \begin{equation}
        M_{\infty}\coloneq \sum_{k=1}^{\infty}d_{k}=\lim_{n\to\infty}M_{n}\label{eq:Mn_limit}
    \end{equation}
    exists with probability one and in $L^{2}$, where
    \begin{equation}
        E[M_{\infty}]=0,\quad E[(M_{\infty})^{2}]=\sum_{k=1}^{\infty}E[(d_{k})^{2}]>0.\nonumber
    \end{equation}
    The conditions of Theorem 1 in \cite{Heyde1977MartingaleLimitTheorem} hold for $\{\hat{s}_{n}^{2}\}$, $\{\hat{V}_{n}^{2}\}$ and $\{\hat{U}_{n}^{2}\}$. Thus, we have Theorem \ref{thm:Mn} \ref{thm:Mn_super}.
\end{proof}
\subsection{An expression of $\{A_n\}$ in terms of $\{S_n\}$ and $\{T_n\}$}\label{subsec:An}
The following lemma together with limit theorems for $\{M_{n}\}$ and $\{T_{n}\}$ yields limit theorems for $\{S_{n}\}$.
\begin{lemma}\label{lemma:An}
    There is a sequence of random variable $\{R_{n}\}$ such that
    \begin{equation}
        A_{n}=\frac{\alpha}{\gamma}\left(S_{n}-\frac{T_{n}}{n^{\gamma}}\right)+R_{n}\label{eq:relation_Sn_Tn_Mn}
    \end{equation}
    with $|R_{n}|\le K$ a.s. for some positive constant $K=K(\alpha,\beta,\gamma)$.
\end{lemma}
\begin{proof}
    By \eqref{eq:Xnconditionalexp} and \eqref{eq:hatSn_decomposition}, we obtain
    \begin{align}\label{eq:A_n}
        A_{n}&=\sum_{k=1}^{n}\frac{E[X_{k}\mid \mathcal{F}_{k-1}]}{k^{\gamma}}=\beta+\alpha\sum_{k=1}^{n-1}\frac{T_{k}}{k(k+1)^{\gamma}}.
    \end{align}
    Since $|T_{k}|\le k$ a.s., we have, with probability one,
    \begin{align}
        \left|\alpha\sum_{k=1}^{n-1}\frac{T_{k}}{k(k+1)^{\gamma}}-\alpha\sum_{k=1}^{n}\frac{T_{k}}{k^{\gamma+1}}\right|&\le |\alpha|\cdot \left|\sum_{k=1}^{n-1}\frac{T_{k}}{k}\left(\frac{1}{(k+1)^{\gamma}}-\frac{1}{k^{\gamma}}\right)\right|+|\alpha|\cdot \frac{|T_{n}|}{n^{\gamma+1}}\nonumber\\
        &\le|\alpha|\sum_{k=1}^{n-1}\frac{|T_{k}|}{k} \left(\frac{1}{k^{\gamma}}-\frac{1}{(k+1)^{\gamma}}\right)+|\alpha|\le 2|\alpha|. \label{eq:mean}
    \end{align}
    Let
    \begin{equation}
        {\sigma}_{l}\coloneq \sum_{k=l}^{\infty}\frac{1}{k^{\gamma+1}}\quad \text{and}\quad {J}_{l}\coloneq \int_{l}^{\infty}\frac{dx}{x^{\gamma+1}}=\frac{1}{\gamma l^{\gamma}}.\nonumber
    \end{equation}
    Rearranging the sum, we have
    \begin{align}
        \sum_{k=1}^{n}\frac{T_{k}}{k^{\gamma+1}}&=\sum_{k=1}^{n}\sum_{l=1}^{k}\frac{X_{l}}{k^{\gamma+1}}=\sum_{l=1}^{n}X_{l}\sum_{k=l}^{n}\frac{1}{k^{\gamma+1}}=\sum_{l=1}^{n}X_{l}\sigma_{l}- T_{n}\cdot\sigma_{n+1}\nonumber\\
        &=\frac{1}{\gamma}S_{n}-\frac{T_{n}}{\gamma n^{\gamma}}+\sum_{l=1}^{n}X_{l}(\sigma_{l}-J_{l})+T_{n}(J_{n}-\sigma_{n+1}).\nonumber \label{eq:rearrangesum}
    \end{align}
    Since $\displaystyle \sigma_{l}\ge J_{l}\ge \sigma_{l+1}$, we obtain
    \begin{equation}
        0\le\sigma_{l}-J_{l}\le\frac{1}{l^{\gamma+1}}\quad \text{and}\quad 0\le J_{l}-\sigma_{l+1}\le \frac{1}{l^{\gamma+1}}.\nonumber
    \end{equation}
    Thus, we have
    \begin{equation}
        \left|\sum_{l=1}^{n}X_{l}(\sigma_{l}-J_{l})\right|\le \sum_{l=1}^{n}\left|X_{l}(\sigma_{l}-J_{l})\right|\le \sum_{l=1}^{\infty}\frac{1}{l^{\gamma+1}}=\sigma_{1},\nonumber
    \end{equation}
    and
    \begin{equation}
        |T_{n}(J_{n}-\sigma_{n+1})|\le \frac{|T_{n}|}{n^{\gamma+1}}\le 1.\nonumber
    \end{equation}
    Therefore, letting
    \begin{equation}
        R_{n}\coloneq \beta + \alpha\!\left(\sum_{k=1}^{n-1}\frac{T_{k}}{k(k+1)^{\gamma}}-\sum_{k=1}^{n}\frac{T_{k}}{k^{\gamma+1}} + \sum_{l=1}^{n}X_{l}(\sigma_{l}-J_{l})+T_{n}(J_{n}-\sigma_{n+1})\right)\!,\nonumber
    \end{equation}
    we obtain \eqref{eq:relation_Sn_Tn_Mn}, where $|R_{n}|\le |\beta|+ (3+\sigma_1)|\alpha|$ almost surely.
\end{proof}
\subsection{Proof of Theorem \ref{thm:Limittheorems}}\label{subsec:ProofofThm2}
By \eqref{eq:hatSn_decomposition} and Lemma \ref{lemma:An}, we obtain
\begin{equation}
    \left|\left(1-\frac{\alpha}{\gamma}\right)S_{n} - \left(M_{n}-\frac{\alpha T_{n}}{\gamma n^{\gamma}}\right)\right|\le K \quad \text{a.s.}\label{eq:Snineq}
\end{equation}
Throughout this subsection, we assume that $\gamma\neq \alpha$.

\ref{enum:Thm2_alphasub}
Suppose that $\alpha\in[-1,1/2)$. Generally, for real sequences $\{x_n\}$ and $\{y_n\}$, we have
\begin{equation}
    \limsup_{n\to\infty} x_{n}+\liminf_{n\to\infty} y_{n}\le \limsup_{n\to\infty} (x_{n}+y_{n})\le \limsup_{n\to\infty} x_{n} +\limsup_{n\to\infty} y_{n}\nonumber
\end{equation}
whenever LHS and RHS of the inequality are well-defined. Using the above inequality, if $\gamma\in (0,1/2)$, then we have
\begin{gather}
    \limsup_{n\to\infty} \pm \frac{M_{n}-\alpha T_{n}/(\gamma n^{\gamma})}{\sqrt{2n^{1-2\gamma}\log\log n}}\ge \left|\frac{1}{\sqrt{1-2\gamma}}-\frac{\alpha}{\gamma\sqrt{1-2\alpha}}\right| \quad \text{a.s.},\nonumber
\end{gather}
which is positive unless $\gamma=-\alpha/(1-2\alpha)$, and
\begin{gather}
    \limsup_{n\to\infty} \pm \frac{M_{n}-\alpha T_{n}/(\gamma n^{\gamma})}{\sqrt{2n^{1-2\gamma}\log\log n}}\le \frac{1}{\sqrt{1-2\gamma}}+\frac{\alpha}{\gamma\sqrt{1-2\alpha}} \quad \text{a.s.},\nonumber
\end{gather}
by \eqref{eq:deffisive_LIL} and Theorem \ref{thm:Mn} \ref{thm:Mn_sub}. 

\noindent\ref{enum:Thm2_alpahsub_gammasub} If $\gamma\in (0,\alpha)$, then, with probability one,
\begin{equation}
    \frac{\alpha}{\gamma\sqrt{1-2\alpha}}-\frac{1}{\sqrt{1-2\gamma}}\le \frac{\alpha-\gamma}{\gamma}\limsup_{n\to\infty} \frac{\pm S_{n}}{\sqrt{2n^{1-2\gamma}\log\log n}}\le \frac{1}{\sqrt{1-2\gamma}}+\frac{\alpha}{\gamma\sqrt{1-2\alpha}}.\nonumber
\end{equation}
In a similar way, for $\gamma\in (\alpha, 1/2)$, we obtain, with probability one,
\begin{equation}
    \frac{1}{\sqrt{1-2\gamma}}-\frac{\alpha}{\gamma\sqrt{1-2\alpha}}\le \frac{\gamma-\alpha}{\gamma}\limsup_{n\to\infty} \frac{\pm S_{n}}{\sqrt{2n^{1-2\gamma}\log\log n}}\le \frac{1}{\sqrt{1-2\gamma}}+\frac{\alpha}{\gamma\sqrt{1-2\alpha}}.\nonumber
\end{equation}
\noindent
\ref{enum:Thm2_alphasub_gammacritical} Suppose that $\gamma=1/2$. By \eqref{eq:deffisive_LIL}, we have
\begin{equation}
    \limsup_{n\to\infty}\pm \frac{T_{n}/n^{1/2}}{\sqrt{\log n}}=\limsup_{n\to\infty}\pm \frac{T_{n}}{\sqrt{n\log n}}=0 \quad \text{a.s.}\nonumber
\end{equation}
Thus, the LIL for $\{S_{n}\}$ follows from Theorem \ref{thm:Mn} \ref{thm:Mn_critical}. Moreover, by \eqref{eq:Snineq} and Theorem \ref{thm:Mn} \ref{thm:Mn_critical}, we have
\begin{equation}
    \left|\frac{1/2-\alpha}{1/2}\frac{S_{n}}{\sqrt{\log n}}-\frac{M_{n}}{\sqrt{\log n}}\right|\to 0 \quad \text{a.s.},\nonumber
\end{equation}
which implies the CLT for $\{S_{n}\}$.

\ref{enum:Thm2_alphacritical} Assume that $\alpha=1/2$ and $\gamma\in (0,1/2)$. We obtain
\begin{equation}
    \limsup_{n\to\infty} \pm \frac{M_{n}}{\sqrt{2n^{1-2\gamma}\log n}}=0 \quad \text{a.s.}\nonumber
\end{equation}
by Theorem \ref{thm:Mn} \ref{thm:Mn_sub}. Therefore, by \eqref{eq:critical_LIL}, we obtain the LIL for $\{S_{n}\}$. In addition, by \eqref{eq:Snineq}, we also have
\begin{equation}
    \left|\frac{1/2-\gamma}{\gamma}\frac{S_{n}}{\sqrt{n^{1-2\gamma}\log n}}-\frac{T_{n}}{2\gamma\sqrt{n\log n}}\right|\to 0\quad \text{a.s.},\nonumber
\end{equation}
which implies the CLT for $\{S_{n}\}$.

\ref{enum:Thm2_alphasuper} We consider the case \(\alpha >1/2\). Let $L$ be the random variable defined by \eqref{eq:superdiffusive_almost_sure_convergence}. Since
\begin{center}
    $\displaystyle \sum_{k=1}^{n}\frac{1}{k^{\gamma-\alpha+1}}\sim \frac{1}{\alpha-\gamma}n^{\alpha-\gamma}$ if $\gamma<\alpha$,\quad and \quad$\displaystyle \displaystyle \sum_{k=1}^{n}\frac{1}{k^{\gamma-\alpha+1}}\sim \log n$ if $\gamma=\alpha$
\end{center}
as \(n\to \infty\), we have, with probability one,
    \begin{center}
        $\displaystyle \sum_{k=1}^{n}\frac{T_{k}}{k^{\gamma+1}}\sim \frac{L}{\alpha-\gamma}n^{\alpha -\gamma}$ if $\gamma<\alpha$,\quad and \quad $\displaystyle \sum_{k=1}^{n}\frac{T_{k}}{k^{\gamma+1}}\sim L\log n$ if $\gamma=\alpha$.
    \end{center}
    Therefore, we obtain, with probability one,
    \begin{center}
        $\displaystyle A_{n}\sim \frac{\alpha L}{\alpha-\gamma}n^{\alpha -\gamma}$ if $\gamma<\alpha$,\quad and\quad $A_{n}\sim \alpha L\log n$ if $\gamma=\alpha$.
    \end{center}
Thus, using \eqref{eq:relation_Sn_Tn_Mn}, the asymptotic behavior of $\{S_{n}\}$ is the same as $\{A_{n}\}$. Rearranging \eqref{eq:Snineq}, we have
\begin{equation}
\left|\frac{\alpha-\gamma}{\gamma}\left(S_{n}-\frac{\alpha L}{\alpha-\gamma}n^{\alpha-\gamma}\right)-\left(\frac{\alpha}{\gamma}\cdot \frac{T_{n}-Ln^{\alpha}}{n^{\gamma}}-M_{n}\right)\right|\le K\quad \text{a.s.}\nonumber
\end{equation}
If $\gamma\in (0,1/2)$, then we get, with probability one,
\begin{equation}
    \limsup_{n\to\infty} \frac{- M_{n}+\alpha(T_{n}-Ln^{\alpha})/(\gamma n^{\gamma}) }{\sqrt{2n^{1-2\gamma}\log\log n}}\ge \frac{1}{\sqrt{1-2\gamma}}-\frac{\alpha}{\gamma\sqrt{2\alpha-1}}=\frac{\gamma\sqrt{2\alpha-1}-\alpha \sqrt{1-2\gamma}}{\gamma\sqrt{(1-2\gamma)(2\alpha-1)}},\nonumber
\end{equation}
which is positive unless $\gamma=-(\alpha+\alpha\sqrt{\alpha^{2}+2\alpha-1})/(2\alpha-1)$. Moreover, if $\gamma=1/2$, then we have, with probability one,
\begin{equation}
    \limsup_{n\to\infty} \pm \frac{- M_{n}+\alpha(T_{n}-Ln^{\alpha})/(\gamma n^{\gamma}) }{\sqrt{2\log n \log\log\log n}}=\limsup_{n\to\infty} \pm \frac{-M_{n}}{\sqrt{2\log n\log\log\log n}}=1.\nonumber
\end{equation}
If $\gamma\in (1/2,\alpha)$, then $\{M_{n}\}$ converges a.s. and $(T_{n}-Ln^{\alpha})/n^{\gamma}\to 0$ a.s., which implies Theorem \ref{thm:Limittheorems} \ref{enum:Thm2_alphasuper} \ref{enum:Thm2_alphasuper_gammasub} and \ref{enum:Thm2_alphasub_gammacritical}. The proof of Theorem \ref{thm:Limittheorems} \ref{enum:Thm2_alphasuper} \ref{enum:Thm2_alphasuper_gammasuper} is postponed to the next subsection. \qed
\subsection{Proof of Theorem \ref{thm:main}}\label{subsec:proof_main}
Note that \eqref{eq:subdiverge} and \eqref{eq:superdiver} follow from Theorem \ref{thm:Limittheorems}. Thus, we concentrate on the case where $\{S_{n}\}$ converges. By Theorem \ref{thm:Mn}, $\{M_{n}\}$ converges with probability one if and only if $\gamma>1/2$. We consider $\{A_{n}\}$.
    Suppose that \(\alpha \le 1/2\). If \(\gamma  > 1/2\), then {\color{black}$\sum_{k=1}^{n}\frac{T_{k}}{k^{\gamma+1}}$} is absolutely convergent almost surely. Indeed, from the LIL for $\{T_{n}\}$ (see \eqref{eq:deffisive_LIL} and \eqref{eq:critical_LIL}), we can deduce that if $\alpha\le 1/2$, then $\displaystyle \lim_{n\to \infty} \frac{T_{n}}{\sqrt{n}\log n}=0$ a.s. Thus, with probability one, there exists a positive constant $C_{1}$ such that
\begin{equation}
    \sum_{k=1}^{n}\frac{|T_{k}|}{k^{\gamma+1}}\le \sum_{k=1}^{n}\frac{C_{1}\sqrt{k}\log k}{k^{\gamma+1}}=C_{1}\sum_{k=1}^{n}\frac{\log k}{k^{\gamma+1/2}}.\nonumber
\end{equation}
It follows from \eqref{eq:A_n} and \eqref{eq:mean} that $\{A_{n}\}$ converges a.s. if $\alpha\le 1/2$ and $\gamma>1/2$. Therefore, if $\alpha\le 1/2$ and $\gamma>1/2$, then $\{S_{n}\}$ converges a.s.

In the case $\alpha>1/2$ and $\gamma>\alpha$, by \eqref{eq:superdiffusive_almost_sure_convergence}, there exists a positive random variable \(C_{2}\) such that
    \begin{equation}
        \sum_{k=1}^{n}\frac{|T_{k}|}{k^{\gamma+1}}\le C_{2}\sum_{k=1}^{n}\frac{1}{k^{\gamma-\alpha +1}} \quad \text{a.s.}\nonumber
    \end{equation}
    Thus, \(\{A_{n}\}\) converges almost surely. Let $\displaystyle A_{\infty}\coloneq \lim_{n\to\infty} A_{n}$ a.s. Since there exists a positive random variable $C_{3}$ such that
    \begin{equation}
        \left|A_{\infty}-A_{n}\right|=\left|\sum_{k=n+1}^{\infty} \frac{T_{k}}{k(k+1)^{\gamma}}\right|\le C_{3} n^{\alpha-\gamma}\quad \text{a.s.,}\nonumber
    \end{equation}
    we have \eqref{eq:superdiffusive_gammasuper}. The martingale part $\{M_{n}\}$ also converges a.s. and $M_{n}-M_{\infty}=o(n^{\alpha-\gamma})$ a.s. by Theorem \ref{thm:Mn} \ref{thm:Mn_super}. This completes the proof of Theorem \ref{thm:main} and Theorem \ref{thm:Limittheorems} \ref{enum:Thm2_alphasuper} \ref{enum:Thm2_alphasuper_gammasuper}.\qed
\section*{Acknowledgements}
I am very grateful to my supervisor Masato Takei for insightful discussions. Furthermore, I would like to extend my thanks to Professor Hideki Tanemura for helpful advises which enable me to improve the result in an earlier draft. Last but not least, I thank reviewers for constructive and helpful comments.

\begin{appendix}
    \section{$L^{2}$-convergence for $\gamma>\gamma_{c}(\alpha)$}\label{sec:A1}
\begin{theorem}\label{thm:Sn_L2convergent}
  If $\alpha\in [-1,1)$ and $\gamma>\gamma_{c}(\alpha)$, then $\displaystyle \lim_{n\to\infty} S_{n}=S_{\infty}$ in $L^{2}$.
\end{theorem}
\begin{proof}
  By Theorem \ref{thm:Mn} \ref{thm:Mn_super}, if $\gamma>1/2$, then $M_{n}\to M_{\infty}$ in $L^{2}$. We show that $A_{n}\to A_{\infty}$ in $L^{2}$ if $\gamma>\gamma_{c}(\alpha)$. By Fatou's lemma, for each $n$, we have
  \begin{equation}
    E[(A_{\infty}-A_{n})^{2}]\le \liminf_{s\to\infty}E[(A_{n+s}-A_{n})^{2}].\label{eq:fatou_ineq}
  \end{equation}
  If $l<m$, then
  \begin{equation}
    E[T_{l}T_{m}]=E[T_{l}E[T_{m}\mid \mathcal{F}_{m-1}]]=\!\left(1+\frac{\alpha}{m-1}\right)E[T_{l}T_{m-1}]=\cdots = \frac{a_{m}}{a_{l}}E[(T_{l})^{2}],\label{eq:correlation}
  \end{equation}
  where $a_{n}$ is defined by \eqref{eq:a_n}. Therefore, by \eqref{eq:correlation},
  \begin{align}
    E[(A_{n+s}-A_{n})^{2}]&=E\left[\alpha^{2}\left(\sum_{k=n}^{n+s-1}\frac{T_{k}}{k(k+1)^{\gamma}}\right)^{2}\right]\nonumber\\
    &=\alpha^{2} \sum_{k=n}^{n+s-1}\frac{E[(T_{k})^{2}]}{k^2(k+1)^{2\gamma}}+2\alpha^{2}\sum_{l=n}^{n+s-2} \sum_{m=l+1}^{n+s-1} \frac{a_{m}}{a_{l}} \frac{E[(T_{l})^{2}]}{l(l+1)^{\gamma}m(m+1)^{\gamma}}\nonumber\\
    &=\alpha^{2} \sum_{k=n}^{n+s-1}\frac{E[(T_{k})^{2}]}{k^2(k+1)^{2\gamma}} + 2\alpha^{2}\sum_{l=n}^{n+s-2}b_{l} \sum_{m=l+1}^{n+s-1} c_{m},\label{eq:A_n_secondmoment}
  \end{align}
  where
  \begin{equation}
    b_{l}\coloneq \frac{\Gamma(l)}{\Gamma(l+\alpha)}\cdot \frac{E[(T_{l})^{2}]}{l(l+1)^{\gamma}}, \quad c_{m}\coloneq \frac{\Gamma(m+\alpha)}{\Gamma(m)}\cdot\frac{1}{m(m+1)^{\gamma}}.\nonumber
  \end{equation}
  It follows from \eqref{eq:T_n_second_moment} that $\sum_{k=1}^{\infty}\frac{E[(T_{k})^{2}]}{k^2(k+1)^{2\gamma}}<+\infty$. Since $\displaystyle c_{m} \sim \frac{1}{m^{1+\gamma-\alpha}}$ as $m\to\infty$, we can find $K>0$ such that $\sum_{m=l+1}^{\infty} c_{m}\le Kl^{-(\gamma-\alpha)}$ for any $l$. Thus, the second term in \eqref{eq:A_n_secondmoment} is bounded by $2K\alpha^{2}\sum_{l=n+1}^{n+s-2}b_{l}l^{-(\gamma-\alpha)}$. Using \eqref{eq:T_n_second_moment}, it is straightforward to see that $\sum_{l=1}^{\infty} b_{l}l^{-(\gamma-\alpha)}<+\infty$. By \eqref{eq:fatou_ineq}, we have $\displaystyle \lim_{n\to\infty} E[(A_{\infty}-A_{n})^{2}]=0$.
\end{proof}
As a consequence of the above theorem, we obtain
\begin{equation}
    E[S_{\infty}]=\lim_{n\to\infty} E[S_{n}]=\beta+\frac{\alpha\beta}{\Gamma(1+\alpha)}\sum_{k=1}^{\infty} \frac{\Gamma(k+\alpha)}{k!(k+1)^{\gamma}}.\nonumber
\end{equation}
Similarly, we can obtain an expression of $E[(S_{\infty})^{2}]$, which looks very complicated and is omitted here.
\end{appendix}

\end{document}